\documentclass[12pt]{amsart}
\usepackage[utf8]{inputenc}
\usepackage{times}
\usepackage{amsmath, amssymb, amsthm, amsfonts, verbatim,fullpage}
\usepackage{amscd}
\usepackage{graphicx}
\usepackage[all]{xy}
\usepackage{enumerate}
\usepackage{graphicx}

\newcommand{\ffi}{\varphi}

\newcommand{\N}{\mathbb{N}}

\newcommand{\PP}{\mathbb{P}}

\newcommand{\Z}{\mathbb{Z}} 
\newcommand{\bb}[1]{\mathbb{#1}} 
\newcommand{\Cc}{\mathcal{C}}

\newcommand{\Ic}{\mathcal{I}}

\DeclareMathOperator{\rk}{rk}

\DeclareMathOperator{\seg}{Seg}
\DeclareMathOperator{\spn}{Span}
\DeclareMathOperator{\cl}{Cl}
\DeclareMathOperator{\supp}{Supp}
\DeclareMathOperator{\reg}{reg}

\newtheorem{theorem}{Theorem}[section]
\newtheorem{proposition}[theorem]{Proposition}
\newtheorem{lemma}[theorem]{Lemma}
\newtheorem{corollary}[theorem]{Corollary}

\theoremstyle{definition}
\newtheorem{remark}[theorem]{Remark}
\newtheorem{define}[theorem]{Definition}
\newtheorem{example}[theorem]{Example}

\numberwithin{equation}{section}

\begin{document}

\title{Segre's Regularity  Bound for Fat Point Schemes}

\author{Uwe Nagel}
\address{Department of Mathematics\\
University of Kentucky\\
715 Patterson Office Tower\\
Lexington, KY 40506-0027 USA}
\email{uwe.nagel@uky.edu}

\author{Bill Trok}
\address{Department of Mathematics\\
University of Kentucky\\
715 Patterson Office Tower\\
Lexington, KY 40506-0027 USA}
\email{william.trok@uky.edu}

\thanks{The first author was partially supported by Simons Foundation grant \#317096.}

\begin{abstract}
Motivated by questions in interpolation theory and on linear systems of rational varieties, one is interested in upper bounds for the Castelnuovo-Mumford regularity of arbitrary subschemes of fat points. An optimal upper bound, named after Segre, was conjectured by Trung and, independently, by Fattabi and Lorenzini. It is shown that this conjecture is true. Furthermore, an alternate regularity bound is established that improves the Segre bound in some cases. Among the arguments is a new partition result for matroids.  
\end{abstract} 


\maketitle

\section{Introduction} 

Given $s$ distinct points $P_1,\ldots,P_s$ of projective space and positive integers $m_1,\ldots,m_s$, we consider homogeneous polynomials that vanish at $P_i$ to order $m_i$ for  $i = 1,\ldots,s$. Equivalently, these are the polynomials such that  $P_i$ is a root of all partial derivatives of order less than $m_i$ for all $i$. The set of all these polynomials is the homogeneous $I_X$ of the fat point scheme $X = \sum_{i=1}^s m_i P_i$. The vector space dimension of the degree $d$ polynomials in $I_X$ is known if $d$ is large. In geometric language, the fat points scheme $X$ imposes independent conditions on forms of degree $d \gg 0$. The least integer $d$ such that this is true for degree $d$ forms  is called the \emph{regularity index} of $X$, denoted $r (X)$. It was  conjectured by Trung (see \cite{T-00}) and, independently, by Fatabbi and Lorenzini in \cite{FL}  that $r(X) \le \seg X$, where $\seg X$ is 
\[
\seg X  := \max\left\{ \left\lceil\frac{ - 1 + \sum_{P_i \in L} m_i}{\dim L} \right\rceil \mid L 
  \subseteq \bb{P}^n \text{ a positive-dimensional linear subspace}\right\}. 
\]

The number $\seg X$ (see also Remark \ref{rem:equiv fraction}) is called the \emph{Segre bound}  because B.~Segre \cite{S} proved the conjecture in the case where the given points are in a projective plane and no three of them are collinear. Segre's result was extended to $\PP^n$ under the assumption that  the given points $P_1,\ldots,P_s \in \PP^n$ are in linearly general position, that is, any subset of $n+1$ of these points spans $\PP^n$ (see \cite{CTV}). Without this assumption, the conjecture has been shown in rather few cases, namely  
\begin{itemize} 
\item for any fat point subscheme of $\PP^2$ in  \cite{F} and \cite{T-99},  independently,
\item  for any fat point subscheme of $\PP^3$ in \cite{FL} and \cite{T-00},  independently,  \; and 
\item if  $s \le n+3$  and the $s$ points span  $\PP^n$ in \cite{BDP}. 
\end{itemize}
Furthermore, there are partial results for certain fat point subschemes of $\PP^4$ (see \cite{B1, B2}) and for some fat point subschemes of $\PP^n$ supported at at most $2n-1$ points (see \cite{CFL}). In this paper we establish the conjecture in full generality, that is, we show $r(X) \le \seg X$ for each fat point subscheme $X$ of some projective space. This bound cannot be improved in general (see Corollary \ref{cor:sharpness}). 

Bounding the regularity index of a fat point scheme $X$ is equivalent to bounding its Castelnuovo-Mumford regularity
\[
\reg (X) = \min \{m \in \Z \; | \; H^1 (\PP^n, \Ic_X (m-1)) = 0\},  
\]
where $\Ic_X$ is the ideal sheaf of $X$, because $r (X) =  \reg (X) - 1$ (see, e.g., Lemma \ref{lem:reg subscheme}).  
Thus,  by  \cite[Theorem 4.1]{Ei} our results have consequences for interpolation problems. 

If the points $P_1,\ldots,P_s$ are generic, then one expects better bounds for the regularity index. Indeed, for generic points a naive dimension count suggests the precise value of the regularity index. In \cite{AH-00},  Alexander and Hirschowitz showed that this naive count is correct in sufficiently large degrees. Moreover, if all points have multiplicity two they completely classified the exceptions  in \cite{AH-95}. In all other cases, similarly complete results are not known. In contrast, the Segre bound is true for any fat point scheme. Moreover, we establish an alternate regularity bound  (see Proposition~\ref{prop:modified Segre bound}) that improves Segre's bound considerably in some cases. In particular, this is true if many of the points in the support are generic (see Example~\ref{exa:mod bound}). 
 
Let us briefly describe the organization of this paper. In Section 2, a crucial new result on matroid partitions is established. Section 3 discusses refinements of a classical tool, the use of residual 
subschemes. Both sets of techniques are first combined in order to establish Segre's bound for reduced zero-dimensional schemes. This is carried out  in Section 4. The arguments in the case of arbitrary fat point schemes are considerably more involved. This is the subject of Section 5. There,  also the optimality and a modification of the Segre bound are discussed. 


\section{Matroid Partitions} 

The goal of this section is to establish a result on matroids that will be a key ingredient for our results on the regularity of a fat point scheme. In order to make the paper accessible to a wide audience we recall some basic facts on matroid. For details we refer to \cite{O}. 

A \emph{matroid} $M$ on a finite \emph{ground set} $E$ is a family of subsets of $E$, called \emph{independent sets}, that is closed under inclusion, that is, any subset of an independent set is independent, and has the additional property that all maximal independent subsets of any subset $A \subseteq E$ have the same cardinality. This maximum cardinality is called the \emph{rank} of $M$, denoted $\rk (M)$. More generally, the \emph{rank} of any subset $A$ of $E$ is the maximum cardinality of an independent subset of $A$. It is denoted by  $\rk_M (A)$ or simply $\rk (A)$ if the matroid $M$ is understood. Equivalently, a matroid on $E$ can be described by means of a function $r_M: 2^E \to \N_0$, which has the following three properties: (i) $0 \le r_M (A) \le |E|$ for all $A$; (ii) $r_M (A) \le r_M (B)$ if $A \subseteq B$; and (iii) $r_M (A \cap B) + r_M (A \cup B) \le r_M (A) + r_M (B)$ for all $A, B \subseteq E$. Then the subsets $I$ of $E$ with $r_M (I) = |I|$ are the independent subsets of a matroid $M$ and $r_M$ is called the 
\emph{rank function}  of $M$. 

The \emph{closure} or \emph{span} of a subset $A \subseteq E$ is the set 
\[
\cl_M (A) = \{e \in E \: | \: \rk (A+e) = \rk (A)\},  
\]
where we use the simplified notation $A + e = A \cup \{e\}$. Similarly, we write $C - e$ for $C \setminus \{e\}$. 

We will discuss partitions of a ground set into independent sets. The following characterization is due to Edmonds and  Fulkerson \cite[Theorem 1c]{EF}. 
 
\begin{theorem}
    \label{thm:ext edmonds}
Given  matroids $M_1,\ldots,M_k$ on a ground set $E$ with rank functions $\rk_1,\ldots,\rk_k$, there is a partition  $E = I_1 \sqcup \cdots \sqcup I_k$ such that each set $I_j$ is independent in  $M_j$   if and only if, for each subset $A \subseteq
  E$, one has $|A| \le \sum_{j=1}^k  \rk_j (A)$.
\end{theorem}  

If all matroids  are equal, one obtains the following earlier criterion by Edmonds \cite{E}.

\begin{corollary}
    \label{cor:edmonds}
Given a matroid, there is a partition of its ground set
  $E$ into $k$ independent sets if and only if, for each subset $A \subseteq
  E$, one has $|A| \le \rk(A) \cdot k$.
\end{corollary} 

Strengthening the assumption, one can find a partition with additional properties.

\begin{theorem} 
     \label{thm:main matroid}
Let $\tilde{M}$ be a matroid on $\tilde{E} \neq \emptyset$, and let $k$ and $p$  be non-negative integers. Assume there is a subset $E \neq \emptyset$ of $\tilde{E}$ such that  
\begin{equation*}
     \label{eq: mod card estimate}
|A| \le k \cdot \rk_{\tilde{M}} A - p
\end{equation*}  
for each non-empty subset $A \subseteq E$, and fix an integer $q$ with $0 \le q \le p$.  Then, for each $q$-tuple $(e_1,\ldots,e_q) \in \tilde{E}^q$, there are disjoint independent sets $\tilde{I}_1,\ldots,\tilde{I}_q$ of $E$  with the following  property: If $(a_1,\ldots,a_p) \in \tilde{E}^p$  is a $p$-tuple whose first $q$ entries are $e_1,\ldots,e_q$, that is, $a_i = e_i$ if $1 \le i \le q$, then there is a partition $E = I_1 \sqcup \cdots \sqcup I_k$   into independent sets  such that $a_j \notin \cl (I_j)$ whenever $1 \le j \le p$ and $I_j = \tilde{I}_j$ for $j = 1\ldots,q$. 
\end{theorem}

Note that in the case where $\tilde{E} = E$ and $p = 0$ this is just Corollary \ref{cor:edmonds}. 
The proof of Theorem~\ref{thm:main matroid} requires some preparation. 
Recall that matroids can also be characterized by their circuits. A \emph{circuit} of a matroid $M$ on $E$ is a minimal dependent subset $C \subseteq E$, that is, $C$ is dependent, but every proper subset of $C$ is independent. 

Fix integers $k, p$ with $k > p \ge 0$ and consider the function $f: 2^E \to \Z$ defined by 
\[
f (A) = k \cdot \rk (A) - p. 
\]
Moreover, let 
\[
\Cc (f) = \{ C \subseteq E \; | \; \emptyset \neq C \text{ is minimal with } f(C) < |C|\}. 
\]
By \cite[Proposition 12.1.1]{O}, there is a matroid on $E$ whose circuits are precisely the elements of $\Cc (f)$. We denote this matroid by $M_{k,p}$ or $M(f)$. Thus, a non-empty subset $J \subseteq E$ is independent in $M(f)$ if and only if $|J| \le f (J)$. We also need the following observation. 

\begin{lemma}
    \label{lem:prop circuit}
If $C$ is a circuit of $M(f)$ and $e \in C$, then $e \in \cl_M (C)$ and $|C| = k \cdot \rk_M (C) - p +1$.     
\end{lemma}

\begin{proof}
Since $C - e$ is independent we obtain 
\[
|C| >  f(C) \ge f (C-e) \ge |C-e| = |C| -1, 
\]
which forces $|C| - 1 = f (C-e) = f (C) = k \cdot \rk_M (C) - p$, and thus $\rk_M (C) = \rk_M (C-e)$. 
\end{proof}

To simplify notation we will simply write $\rk (A)$ and $\cl (A)$ if these concepts refer to the original matroid $M$. We are ready to  establish a key result. 

\begin{proposition}
    \label{prop:key rank est}
Let $M$ be a matroid on $E \neq \emptyset$, and let $k$ and $p$  be non-negative integers. Assume that 
\begin{equation}
     \label{eq:card estimate ind}
|A| \le (k+1) \cdot \rk A - (p+1)
\end{equation}  
for each non-empty subset $A \subseteq E$. Then the rank of $M_{k, p}$ satisfies 
\[
\rk_{M_{k, p}} (E) \ge |E| - \rk (E) + 1. 
\]
\end{proposition}

\begin{proof}
Notice that by applying Assumption \eqref{eq:card estimate ind} to a set with one element, we get $k > p$. Thus, the matroid $M_{k, p} = M(f)$ with $f (A) = k \cdot \rk (A) - p$ is well-defined. 

Set $r = \rk M = \rk (E)$, and let $I \subseteq E$ be any independent set of $M(f)$. We have to show: If $|E - I | \ge r$, then there is some $b \in E - I$ such that $I + b$ is independent in $M(f)$. 

Suppose on the contrary that there is a subset $B = \{b_1,\ldots,b_r\}$ of $r$ elements in $E - I$ and that, for each $b_i \in B$ the set $I + b_i$ is dependent in $M(f)$. Then, for each $i$, there is a minimal subset $F_i \subset I$ such that $F_i + b_i$ is dependent in $M(f)$. Thus, $F_i + b_i$ is a circuit of $M(f)$.  Using  Lemma~\ref{lem:prop circuit}, we conclude  that, for each $i$, one has 
\[
b_i \in \cl (F_i) \; \text{ and } \; |F_i| = k \cdot \rk (F_i) - p. 
\]
Our next goal is to show the following assertion. 
\smallskip

\noindent 
\emph{Claim:} There are $s \le r$ subsets $A_1,\ldots,A_s$ of $I$ that satisfy the following conditions: 
\begin{itemize}

\item[(i)] $|A_i|  = k \cdot \rk (A_i) - p$; 

\item[(ii)] $B \subset \bigcup_{i=1}^s \cl (A_i)$; \; \text{ and } 

\item[(iii)] $\rk (A_i \cup A_j) = \rk (A_i) + \rk (A_j)$ if $i \neq j$.  
\end{itemize}

We prove this claim recursively. Initially, put $s = r$ and $A_i =
F_i$ for $i = 1,\ldots,r$. Then the set $\{A_1,\ldots,A_s\}$ satisfies
conditions (i) and (ii). Thus, the claim follows once we have shown:
If a set $\{A_1,\ldots,A_s\}$ satisfies conditions (i) and (ii), but
there are elements $A_i$ and $A_j$ with $i \neq j$ and $\rk (A_i \cup
A_j) \neq \rk(A_i)+\rk(A_j)$, then, setting $\hat{A_i} = I \cap \cl(A_i \cup A_j)$, the set 
$\{A_1,...,A_s\} - \{A_i,A_j\} + \hat{A_i}$ also satisfies conditions (i) and (ii). 

Indeed, repeating this process as many times as necessary will result in a collection of subsets of $I$ that satisfies conditions (i) - (iii) because in each step the number of subsets decreases and condition (iii) is trivially satisfied if $s = 1$.  

In order to establish the recursive step, it is enough to show that  $|\hat{A_i}|
= k \cdot \rk(\hat{A_i}) - p$ because $A_i \cup A_j \subseteq \hat{A_i}$ implies that condition (ii) is satisfied for the modified collection. 

To this end  notice that $|A_i \cap A_j| \le f (A_i \cap A_j)$ if $A_i \cap A_j \neq \emptyset$ because
$A_i \cap A_j$ is independent in $M(f)$. It follows that $|A_i
\cap A_j| \leq k [\rk(A_i)+\rk(A_j)-\rk(A_i \cup A_j)]-p$. Observe that this inequality holds even if $A_i \cap A_j = \emptyset$ because by our hypothesis $\rk (A_i) + \rk (A_j) > \rk (A_i \cup A_j)$, and $k > p$. Hence, we  obtain
\begin{align*}
|\hat{A_i}|  \ge |A_i \cup A_j| & = |A_i| + |A_j| - |A_i \cap A_j| \\
 & \ge  \left[k \cdot \rk(A_i) - p\right] + \left[k \cdot \rk(A_j) -
   p\right] \\
& \hspace{.3in} -  \left[k \cdot (\rk(A_i)+\rk(A_j)-\rk(A_i \cup A_j)) - p\right]\\
& =  k  \cdot \rk (A_i \cup A_j) - p \\
& = k \cdot  \rk(\hat{A_i}) - p, 
\end{align*}
 Since $\hat{A_i}$ is independent in $M(f)$, we also have $|\hat{A_i}| \le f (\hat{A_i}) =  k \cdot \rk(\hat{A_i}) - p$, and the desired equality $|\hat{A_i}|
= k \cdot \rk(\hat{A_i}) - p$ follows. Thus, the above claim is shown. 

Now we proceed with the proof of the proposition. Let $A_1,\ldots,A_s$ be a collection of non-empty subsets of $I$ satisfying conditions (i) - (iii) above. Set $B_i = B \cap A_i$. Using $B \subseteq E - I$ and applying the assumption to $A_i \cup B_i$, we get 
\[ 
|A_i| + |B_i| = |A_i \cup B_i| \le (k+1) \cdot  \rk(A_i\cup B_i)-(p+1) =
(k+1) \cdot \rk(A_i) - (p+1). 
\]
Hence condition (i) gives $|B_i| \le \rk (A_i) - 1$. Taking also into
account that the sets $A_1,\ldots,A_s$ are necessarily disjoint, we obtain 
\begin{align*}
|I|  \ge |\bigcup_{i=1}^s A_i| = \sum_{i=1}^s |A_i|  &  = \sum_{i=1}^s [ k \cdot \rk(A_i) - p] \\
& \ge \sum_{i=1}^s [ k \cdot  (|B_i|+1) - p] \\
& \geq k \left(\sum_{i=1}^s |B_i|\right) + s(k-p) = k \cdot  |B| + s(k-p). 
\end{align*}
Since $k > p$ and $|B| = r = \rk (E)$, it follows that 
\[ 
|I| \ge k \cdot \rk (E)  + 1. 
\]
However, this is impossible because 
\[
|I| \le |E - B|  = |E| - |B|  \le  (k+1) \cdot \rk(E) - (p+1) - |B| = k \cdot \rk (E)  
- (p+1).
\]
Thus, the argument is complete. 
\end{proof}

In order to establish  a consequence of this results, we need two particular matroid constructions.  

\begin{define}
   \label{def:ext matroid} Let $M$ be a matroid on $E$. 

(i) Suppose $M$ is a submatroid of a matroid $\tilde{M}$ on $\tilde{E}$. For any $e \in \tilde{E} \setminus E$, define a matroid $M/e$ on $E$  by the rank function $\rk_{M/e} (A) = \rk_{\tilde{M}} (A + e) - 1$ for subsets $A \subseteq E$. It is called an \emph{elementary quotient} of $M$. Note that  the independent sets of $M/e$ are the independent sets of $M$ whose span does not contain $e$. 
   
(ii)  Let $S$ be any subset of $E$. Realize the disjoint union $E \sqcup S$ as $(E, 0) \cup (S, 1)$. Denote by $M_{+S}$ the matroid whose independent sets are of the form $(I_1, 0) \cup (I_2, 1)$ with $\rk_M (I_1 \cup I_2) = |I_1| + |I_2|$. The matroid $M_{+S}$ is called the  \emph{parallel extension of $M$ by $S$}. 
\end{define} 

It is straightforward to check that $M_{+S}$ is indeed a matroid. Its rank is equal to the rank of $M$. More generally, if $A = (A_1,0) \cup (A_2, 1)$ is any subset of $E \sqcup S$, then $\rk_{M_{+S}} (A) = \rk_M (A_1 \cup A_2)$. 

\begin{corollary}
    \label{cor:inductive step} 
Let $\tilde{M}$ be a matroid on $\tilde{E} \neq \emptyset$, and let $M$ be the submatroid induced  on a subset  $E \neq \emptyset$ of $\tilde{E}$. 
Assume that, for non-negative integers  $k$ and $p$ and each non-empty subset $A \subseteq E$, one has 
\begin{equation*}
     \label{eq:card estimate}
|A| \le (k+1) \cdot \rk A - (p+1). 
\end{equation*}  
Then, for any $e \in \tilde{E}$, there is an independent set $I \subset E$ such that $e \notin \cl (I)$ and 
\[
|B| \le k \cdot \rk (B) - p 
\]   
for each non-empty subset $B \subseteq E-I$. 
\end{corollary}

\begin{proof}
Consider the function $f: 2^E \to \Z$ defined by $f (A) = k \cdot \rk (A) - p$, and denote the submatroid of $\tilde{M}$ induced on $E$ by $M$.  

Let  $A \neq \emptyset$ be any subset of $E$. Applying Proposition \ref{prop:key rank est} to the  submatroid of $M$ induced on $A$, we get 
$\rk_{A(f)} (A) \ge |A| - \rk (A) + 1$, and so 
\begin{equation}
   \label{eq:est decomp}
   |A| \le \rk (A) + \rk_{A(f)} (A) - 1 \le \rk (A) + \rk_{M(f)} (A) - 1. 
\end{equation} 
We now consider two cases. 
\smallskip 

\emph{Case 1}: Suppose $e$ is not in $E$. Consider the elementary quotient $M/e$ on $E$. By definition, for each subset $A \subseteq E$, one has $\rk_{M/e} (A) = \rk_{\tilde{M}} (A + e) - 1$. It follows that $\rk_{M/e} (A) \ge \rk (A) -1$. Hence, Equation \eqref{eq:est decomp} gives
\[
|A| \le \rk_{M/e} (A) + \rk_{M(f)} (A). 
\]
Using Theorem \ref{thm:ext edmonds}, we conclude that there is a decomposition $E  = I \sqcup J$ such that $I$ is independent in  $M/e$ and $J$ is independent in $M(f)$. Be definition of $M/e$, the span of $I$ does not contain $e$. Therefore, $E = I \sqcup J$ is a partition with the required properties because,  for each subset $B \neq \emptyset$ of $J$, one has 
\[
|B| \le f(B) = k \cdot \rk (B) - p 
\]
as $J$ is independent in $M(f)$. 
\smallskip 

\emph{Case 2}: 
Suppose $e$ is in  $E$. Then consider first the parallel extension $M_{+\{e\}}$ of $M$ on the set $(E, 0) \cup \{(e, 1)\}$. Second, passing to an elementary quotient of $M_{+\{e\}}$, we get a matroid $M_{+\{e\}}/(e,1)$ on the ground set $(E, 0)$. To simplify notation, let us denote the latter matroid by $M_{+e}/e$ and identify its ground set with $E$. Thus, we get for $A \subseteq E$ that 
\[
\rk_{M_{+e}/e} (A) = \rk_{M_{+\{e\}}} ((A, 0) \cup \{(e, 1)\}) - 1 = \rk (A+e) - 1 \ge \rk (A) - 1.  
\]
Now we conclude as in Case 1, using $M_{+e}/e$ in place of the matroid $M/e$. 
\end{proof}

We are now in a position to establish the  announced partition result.

\begin{proof}[Proof of Theorem \ref{thm:main matroid}]
If $p = 0$, then  the assertion is true by Edmond's criterion (Corollary \ref{cor:edmonds}). 

Let $p \ge 1$. First, we construct a suitable partition for a fixed $p$-tuple $(a_1,\ldots,a_p) \in \tilde{E}^p$ step by step. Consider $a_1 \in E$. By Corollary \ref{cor:inductive step}, there is a partition $E = I_1 \sqcup J_1$ such that $I_1$ is independent in $M$, $e_1 \notin \cl (I_1)$, and $|B| \le (k-1) \cdot \rk (B) - (p-1)$  for each non-empty subset $B \subseteq J_1$. Thus, we are done if  $p = 1$. If $ p \ge 2$, 
we apply Corollary \ref{cor:inductive step} again, this time to $a_2 \in E$ and the  submatroid of $M$ induced on $J_1$. 
After $p$ applications of Corollary \ref{cor:inductive step}, we obtain a partition $E = I_1 \sqcup \ldots  \sqcup I_p \sqcup J_p$  such that $I_1,\ldots,I_p$ are independent in $M$, $a_j$ is not in the span of $I_j$ for each $j$, and $|B| \le (k-p) \cdot \rk (B)$  for each non-empty subset $B \subseteq J_p$. Applying Corollary \ref{cor:edmonds} to the submatroid on $J_p$, we get a partition $J_p = I_{p+1} \sqcup \ldots  \sqcup I_k$ into independent sets of $M$. This produces a desired partition for a fixed $(a_1,\ldots,a_p)$. 
 
Second, we note that in the above construction the first $p$ independent sets are obtained sequentially. Once the sets $I_1,\ldots,I_{j -1}$ have been found, the set $I_j$ is determined in the complement of $I_1 \sqcup \ldots \sqcup I_{j-1}$. It depends on the choice of $a_j$, but not on the elements $a_{j+1},\ldots,a_k$.  This shows in particular that the sets $I_1,\ldots,I_q$ are independent of the elements $a_{q+1},\ldots,a_k$. Thus, the argument is complete. 
\end{proof}

\begin{remark} 
     \label{rem:matroid partition} 
(i) Using the notation of the proof of Corollary~\ref{cor:inductive step}, the  partition result in Theorem \ref{thm:main matroid} can be also stated as follows: There is a partition $E = I_1 \sqcup \cdots \sqcup I_k$ such that $I_{p+1},\ldots,I_k$ are independent in $M$ and, for each $j= 1,\ldots,p$, the set $I_j$ is independent in $M/a_j$ if $a_j \notin E$ and independent in $M_{+a_j}/a_j$ if $a_j \in E$, respectively.

(ii)   If the ground set $E$ of a matroid  can be partitioned into $k$ independent sets, then Edmond's criterion (Corollary \ref{cor:edmonds}) implies that there is an independent set $I$ such that $|A| \le (k-1) \cdot \rk A$ for each subset $A$ of $E \setminus I$. Thus, for a matroid satisfying the assumptions of Theorem \ref{thm:main matroid}, it is natural to wonder if there is an independent set $I$ of $E$ such that,  for each $e \in I$ and each $A \subset (E \setminus I) + e$, one has $|A| \le (k-1) \cdot \rk_{\tilde{M}} A - p$. However, this is not always possible, not even for representable matroids, see Example~\ref{exa:partion optimal}. 
\end{remark} 


\section{Inductive Techniques}

We now begin considering zero-dimensional subschemes of projective space. In this section we collect some facts that are used in subsequent parts of this note. 

Let $K$ be an arbitrary field, and let $X$ be any projective subscheme of some projective space $\PP^n = \PP^n_K$.  For short, we often write $H^1 (\Ic_X (j))$ instead of $H^1 (\PP^n, \Ic_X (j))$ for the first cohomology of its ideal sheaf $\Ic_X$. We use $R = K[x_0,\ldots,x_n]$ to denote the coordinate ring of $\PP^n$. 

\begin{lemma}
    \label{lem:reg subscheme}
 Let $X \subset \PP^n$ be a zero-dimensional subscheme.    
\begin{itemize}
\item[(a)] Then $r (X) = \min \{ j \in \Z \; | \; H^1 (\Ic_X (j)) = 0\}$. 

\item[(b)] For any  zero-dimensional subscheme $Z$ of $X$, one has that $r(Z) \le r(X)$.  
\end{itemize}
\end{lemma} 

\begin{proof} These results are known to specialists. We include a proof for the convenience of the reader. 
Part (a) is a consequence of 
\[
h_X (j) - \deg X = -  \dim_K  H^1 (\Ic_X (j)). 
\]
This relation also shows that $h_X (j) \le \deg X$ for all integers $j$ and that equality is true if and only if $j \ge r(X)$. Hence, the 
exact sequence $0 \to I_Z/I_X \to R/I_X \to R/I_Z \to 0$ gives that $h_X (j) = \deg X$ implies $h_Z (j) = \deg Z$. Now (b) follows. 
\end{proof}

A special case of Lemma \ref{lem:reg subscheme}(b) has been shown in \cite[Proposition 3.2]{T-16}. We also need an extension of \cite[Lemma 1]{CTV}. 

\begin{lemma}
    \label{lem:CTV}
Let $Z \subset \PP^n$ be a zero-dimensional scheme, and let $P \in \PP^n$ be a point that is not in the support of $Z$. Then one has, for every integer $m \ge 1$, 
\[
r (Z + m P) = \max\{ m-1, \; r(Z), \; 1 + \reg (R/(I_Z + I_P^m))\}. 
\]    
\end{lemma}

\begin{proof}
The argument is essentially  given in \cite{CTV}. We recall it for the reader's convenience. 

Consider the Mayer-Vietoris sequence
\[
0 \to R/I_{Z+m P} \to R/I_Z \oplus R/I_{m P} \to R/(I_Z + I_P^m) \to 0. 
\]
Since $\deg (Z + mP) = \deg Z + \deg (mP)$ and $r(mP) = m-1$,  it shows that $h_{Z+ m P} (j) = \deg (Z + mP)$ if and only if $h_Z (j) = \deg Z$, \; $h_{mP} (j) = \deg {mP}$, and $[R/(I_Z + I_P^m)]_j = 0$. 
\end{proof} 

The following result follows from a standard residual sequence (see
\cite[Theorem 3.2]{FL} for a special case). 

\begin{lemma}[Inductive Technique 1] 
    \label{lem:indTech 1}
Let $Z \subset \PP^n$ be a zero-dimensional scheme, and let $F \subset \PP^n$ be a   hypersurface defined by a form $f \in R$. Denote by $\emptyset \neq W \subset \PP^n$ the residual of $Z$ with respect to $F$ (defined by $I_Z : f$). If $Z \cap F \neq \emptyset$, then one has 
\[
r(Z) \le \max\{ r(W) + \deg F, \; r (Z \cap F)\}. 
\]
\end{lemma} 

\begin{proof} 
Let $d = \deg F$. 
Multiplication by $f$ induces the following exact sequence of ideal sheaves
\[
0 \to \Ic_W (-d) \to \Ic_Z \to \Ic_{Z \cap F} \to 0. 
\]
Its long exact cohomology sequence gives, for all integers $j$, 
\[
H^1 (\Ic_W (j-d)) \to H^1 (\Ic_Z (j)) \to H^1 (\Ic_{Z \cap F}(j)). 
\]
Now the claim follows because $r (Z) = \min \{ j \in \Z \; | \; H^1 (\Ic_Z (j)) = 0\}$ (see Lemma  \ref{lem:reg subscheme}). 
\end{proof} 

If a hypersurface $F$ is defined by a form $f$, then we also write $Res_f (Z)$ for $Res_F (Z)$. 

For induction on the multiplicity of a point in the support of a fat point scheme, the  statement below will be useful. 

\begin{lemma}[Inductive Technique 2] 
   \label{lem:indTech 2}
Let $Z \subset \PP^n$ be a zero-dimensional scheme, and let $P \in \PP^n$ be a point that is not in the support of $Z$.   Fix integers $m$ and $k$ with $1 \le k \le m-1$. Set $t = \binom{n-1+k}{k}$. Assume there are homogeneous polynomials $g_1,\ldots,g_t \in R$ and  $f_1,\ldots,f_t \in R$ such that $I_P^k = (g_1,\ldots,g_t)$, \ $f_i (P) \neq 0$, and 
\[
r (Res_{g_i f_i} (Z + mP)) \le b - k - \deg f_i
\]
 for all $i \in \{1,2,\ldots,t\}$ and some integer $b \ge m-1$. If $r(Z + (m-1) P) \leq b$, then $r (Z + m P) \le b$. 
\end{lemma}

\begin{proof}
By Lemma~\ref{lem:CTV}, it is enough to show $[R/(I_Z + I_P^m)]_b = 0$. 
Observe that 
\[
\dim_K [R/(I_Z + I_P^m)]_b = \sum_{j=0}^{m-1} \dim_K [(I_Z + I_P^j)/(I_Z + I_P^{j+1})]_b. 
\]
By assumption and Lemma \ref{lem:reg subscheme}, we know $r(Z + jP) \leq b$ if $0 \le j < m$. Hence Lemma~\ref{lem:CTV} gives $[I_Z + I_P^j]_b = [R]_b$. It follows that 
\[
\dim_K [R/(I_Z + I_P^m)]_b = \dim_K [(I_Z + I_P^{m-1})/(I_Z + I_P^{m})]_b.
\]
Thus, we are done once we have shown
\begin{align}
   \label{eq:ideal equal}
   [I_Z + I_P^{m-1}]_b = [I_Z + I_P^{m}]_b. 
\end{align}
Let $\ell \in R$ be any linear form that does not vanish at $P$. Then $(x_0,\ldots,x_n) = (\ell, I_P)$. Since $I_P^{m-1}$ is generated by polynomials of degree $m-1$, it follows that  Equality~\eqref{eq:ideal equal} is true if and only if 
\begin{align}
   \label{eq:key inclusion}
\ell^{b-m+1} \cdot [I_P^{m-1}]_{m-1} \subset I_Z + I_P^m. 
\end{align}
Observe that, for each $i \in [t] = \{1,2,\ldots,t\}$, the scheme $W_i := Res_{g_i f_i} (Z + m P))$ is defined by $I_{Z + m P} : (g_i f_i)$ and has multiplicity $m-k$ at $P$ because $f_i (P) \neq 0$ and $g_i$ vanishes precisely to order $k$ at $P$ by assumption. Denote by $J_i$ the homogeneous ideal of $W_i - (m-k)P$. Thus, $I_{W_i} = J_i \cap I_P^{m-k}$.  Hence,  Lemma \ref{lem:CTV} gives
\[
r(W_i) = \max\{m-k-1, \; r(W_i - (m-k)P), \; 1 + \reg (R/(J_i + I_P^{m-k}))\}. 
\]
Since  $r(W_i) \le b - k- d_i$ by assumption, where $d_i = \deg f_i$, we get as above, for each $i \in [t]$,  
\[
0 = \dim_K [R/(J_i + I_P^{m-k})]_{b-k-d_i}  = \sum_{j=0}^{m-k-1} \dim_K [(J_i + I_P^j)/(J_i + I_P^{j+1})]_{b-k-d_i}.
\]
In particular, this yields $[J_i + I_P^{m-k-1}]_{b-k- d_i} = [J_i + I_P^{m-k}]_{b-k- d_i}$. We conclude
\begin{align}
  \label{eq:containment}
  \ell^{b-d_i - m + 1} \cdot [I_P^{m-k-1}]_{m-k-1} \subset J_i + I_P^{m-k}  
\end{align}
because $b - d_i - m +1 \ge 0$. The latter estimate follows from $(m-k)P \subset W_i$, which implies $0 \le m-k-1 = r((m-k)P) \le r(W_i) \le b - k - d_i$ (see Lemma \ref{lem:reg subscheme}). 

Note that, for each $i \in [t]$, one has $J_i = I_Z : (g_i f_i)$. Using $g_i \in I_P^k$ this gives 
\[
g_i f_i \cdot (J_i + I_P^{m-k}) \subset I_Z + I_P^m. 
\]
Combined with  Inclusion \ref{eq:containment}, we get 
\[
g_i f_i \ell^{b-d_i - m + 1}  \cdot  [I_P^{m-k-1}]_{m-k-1} \subset I_Z + I_P^m.   
\]

Since $f(P_i) \neq 0$, possibly after rescaling, we may write $f_i = h_i + \ell^{d_i}$ for some $h_i \in I_P$. Substituting, we obtain, 
\[
g_i (h_i + \ell^{d_i}) \ell^{b-d_i - m + 1}  \cdot  [I_P^{m-k-1}]_{m-k-1} \subset I_Z + I_P^m.   . 
\]
Now $g_i h_i \in I_P^{k+1}$ yields 
\[
\ell^{b - m + 1} g_i  \in I_Z + I_P^m \quad \text{ for each } i \in [t]. 
\]
Since $\{g_1,\ldots,g_t\}$ is a $K$-basis of   $[I_P^{k}]_{k}$, this establishes the desired Containment \eqref{eq:key inclusion}.  
\end{proof}


\section{Reduced Zero-dimensional Subschemes}

We now establish the Segre bound for an arbitrary finite sets of  points. To this end we use suitable vector matroids. 

Recall that a vector matroid or representable matroid $M$ over a field $K$ is given by  an $m \times n$ matrix $A$ with entries in $K$. Its ground set  $E$ is formed by the column vectors of $A$, and the rank of a subset of $E$ is the dimension of the subspace of $K^n$ they generate. Here we adapt this idea in order to use it in a projective space instead of an affine space. 

\begin{define} 
     \label{def:vector matroid}
(i) For a point $P$ of $\PP^n$ and an integer $m \ge 1$, denote by $[P]^m$ an $(n+1) \times m$ matrix whose $m$ columns are all equal to a vector $v \in K^{n+1}$, where $v$ is any representative of the point $P$. 

(ii) Let $X = \sum_{i=1}^s m_iP_i \subset \PP^n$ be a fat point scheme. We write $A_X := \oplus_{i=0}^s [P_i]^{m_i}$ for the concatenation of the matrices $[P_i]^{m_i}$. Define the \emph{matroid of $X$} on the column set $E_X$ of $A_X$, denoted $M_X$, as the vector matroid to the  matrix $A_X$. Thus $|V_X| = \sum_{i=1}^s m_i$. 
\end{define}

\begin{remark} 
     \label{rem:coordVector, span}
(i) Since we are only interested in the span of a subset of columns, the above definition does not depend on the choice of coordinate vectors for the points. Abusing notation slightly, we will identify a non-zero vector of $K^{n+1}$ with a point in $\PP^n$. 

(ii) For consistency of notation,  $\rk$ will always refer to rank in the matroid sense, that is, to a dimension of a subspace of $K^{n+1}$,  and
  $\dim$ will always refer to dimension in $\bb{P}^n$.  Hence, if $S$ is a subset of the column set  $E_X$, then  $\rk (S) = 1 + \dim_{\PP^n} S$. Furthermore, we
  will use $\cl$ to refer to the closure operator in a matroid and
  $\spn$ to refer to the span of the points in $\bb{P}^n$.
  
\end{remark}

Recall that the Segre bound of $X = \sum_{i=1}^s m_iP_i$ is 
\[
\seg (X) = \max \left  \{\left\lceil\frac{w_L(X) - 1}{\dim L} \right\rceil \; | \;  L \subseteq \PP^n \text{ a positive-dimensional linear subspace} \right \}, 
\]
where $w_L (X) = \sum_{P_i \in L} m_i$ is the \emph{weight} of $L$. 

\begin{remark}
  \label{rem:equiv fraction} 
In the literature the Segre bound has also been defined as   
\[
\seg (X) = \max \left \{\left\lfloor\frac{w_L(X) + \dim L - 2}{\dim L} \right\rfloor \; | \;  L \subseteq \PP^n \text{ a positive-dimensional linear subspace} \right \}.  
\]
Obviously, this is equivalent to our definition above. 
\end{remark}

\begin{lemma}
    \label{lem:Segre bounds multi} 
If      $X = \sum_{i=1}^{s} m_iP_i$ is a fat point scheme whose support consists of at least two distinct points, then $m_i \leq \seg (X)$ for all $i$ and $\seg (X) \ge m_i + m_j - 1$ whenever $i \neq j$. 
\end{lemma}

\begin{proof}
Let $L$ be a line passing through two distinct points $P_i$ and $P_j$ in the support of $X$. Then $w_L (X) \ge m_i + m_j$, which implies $\seg (X) \ge m_i + m_j - 1$. 
\end{proof} 

\begin{remark}
  \label{rem:Segre single point} 
If $X = m_1 P_1$ is supported at a single point, then $r(X) = \seg X = m_1 - 1$.   
\end{remark} 

The following is the main result of this section. 

\begin{theorem}
    \label{thm:add one point}
Let $Z \subset \PP^n$ be a fat point scheme satisfying $r (Z) \le \seg (Z)$. Then, for every point $P \in \PP^n$ that is not in the support of $Z$,  one has $r(Z + P) \le \seg (Z+P)$. 
\end{theorem}   

\begin{proof}
We want to use inductive technique 1. To this end, consider the matrix 
\[
A = A_Z \oplus [P]^B = \oplus_{i=1}^s [P_i]^{m_i} \oplus [P]^B, 
\]
where $B = \seg (Z + P)$ and $Z = \sum_{i=1}^s m_i P_i$. Let $M$ be the  vector matroid  on the column set $V$ of $A$. Set $X = Z + P$. 

Consider any subset $S$ of $V$. If $P \notin \spn (S)$, then the definition of weight gives 
\[
|\cl(S)| = w_{\spn(S)}(Z) = 
w_{\spn(S)}(X).
\]
If $P \in \spn (S)$, then $w_{\spn(S)}(X) = 1 + w_{\spn(S)}(Z)$, and thus
\[
|\cl(S)| =
w_{\spn(S)}(X) + B - 1.
\]
In either case we have
\[
|S| \leq w_{\spn(S)}(X) + B-1. 
\]
Using $\rk (S) = 1 + \dim_{\PP^n} S$, the definition of $B = \seg (X)$ yields, for any subset $S \subset V$ with $\rk (S) \ge 2$,
\[
\frac{|S| - B}{\rk(S)-1} \leq \frac{w_{\spn(S)}(X)-1}{\dim
  (\spn(S))} \leq \seg(X) = B.
\]
It follows that 
\[
|S| \leq \rk(S) \cdot B. 
\]
This estimate is also true if $\rk (S) \le 1$ as $B \ge m_i$ for all $i$ (see Lemma \ref{lem:Segre bounds multi}). Therefore Corollary~\ref{cor:edmonds} gives that there is a partition of the column set $V$ into $B$ linearly independent subsets $I_1,\ldots,I_B$. Note that $P \in I_j$ for each $j \in \{1,2,\ldots,B\}$ as $B$ columns of the matrix $A$ correspond to the point $P$. Thus, for each such  $j$,  there is a  hyperplane $H_j$ such that 
\[
\spn (I_j \setminus \{P\}) \subset H_j \quad \text{ and } \quad P \notin H_j. 
\]
It follows that the hypersurface $F = H_1 + \cdots + H_B$ does not contain $P$. However, $F$ does contain $Z$ because any form defining $F$ vanishes at each point $P_j$ to order at least $m_j$ 
as $m_j$ columns of $A$ correspond to $P_j$. Hence we get $Res_F (X) = P$ and $X \cap F = Z$. Now Lemma~\ref{lem:indTech 1} gives $r (X) \le \max \{B, r (Z)\} = B$, as desired. \end{proof}

\begin{corollary} 
If $X$ is any reduced zero-dimension subscheme of  $\bb{P}^n$, then
  $r(X) \le \seg(X)$. 
  \end{corollary}

\begin{proof} 
This is true if $X$ consists of one point (see Remark \ref{rem:Segre single point}). Thus, we conclude by induction on the cardinality of $X$  using the above theorem.
\end{proof}

We conclude this section with an example as promised in Remark~\ref{rem:matroid partition}(ii). 

\begin{example}
   \label{exa:partion optimal} 
Consider any integers $k > p > 0$, and let $K$ be an infinite field. Let $L_1,\ldots,L_t \subset K^{t-1}$ be $t$ generic one-dimensional subspaces, where $t \ge \frac{k}{p} + 1$. On each of the lines choose generically $k-p$ points. Let $M$ be the vector matroid on the set $E$ of all these vectors. Then, one has for each non-empty subset $A \subset E$ that $|A| \le k \cdot \rk A - p$. 
Indeed, if $A = E$ this follows because $|E| = t (k - p) \le k \cdot \rk E - p = k  \cdot (t-1) - p$ by the assumption on $t$. If the rank of $A$ is at most $t-2$, then it contains at most $\rk A$ of the lines $L_1,\ldots,L_t$, which implies $|A| \le \rk A \cdot (k -p) \le k \cdot \rk A - p$, as desired. 

Assume now there is an independent $I \subset E$ with at most $t-2$ elements such that for each non-empty subset $B \subset E \setminus I$ one has $|B| \le (k-1) \cdot \rk B - p$. Thus, $|B| \le k-1-p$ if $B$ has rank one. Consider now $B = E \setminus I$. By assumption on $I$, we have $|B| \ge t (k-p) - (t-2) = t (k-p-1) + 2$. However, we also obtain $|B| = \sum_{i=1}^t |B \cap L_i| \le t (k-p-1)$. 
This contradiction shows that $M$ is a matroid as desired in Remark~\ref{rem:matroid partition}(ii). 
\end{example}


\section{Arbitrary Fat Point Schemes}

The goal of this section is to establish the conjecture by Trung, Fattabi, and Lorenzini. We also discuss the sharpness of the Segre bound and establish an alternate regularity estimate. 

We need one more preparatory result on the matroid introduced in Definition \ref{def:vector matroid}. 

\begin{lemma} 
     \label{lem:card estimate}
Consider the vector matroid $M$ to a fat point scheme $Z = \sum_{j=1}^s m_j P_j$ on the column set $E_Z$. Then, for every subset $S \subset E_Z$ with $\rk S \ge 2$, one has 
\[
|S|  \leq \seg (Z) \cdot \{ \rk(S) - 1 \}+1.
\]
\end{lemma}

\begin{proof} Recall that $\rk (S) = \dim (\spn(S))+1$ for any subset $S \subset E_Z$.  Moreover, one has 
   $|S| \le |\cl_{M} (S)| = w_{L}(Z )$, where $L = \spn(S)$.  Hence, if $\rk S \ge 2$  we obtain
\[
\frac{|S|-1}{\rk(S)-1} \leq  \frac{w_L(Z)-1}{\dim L}  \le \seg (Z). 
\]
Now the claim follows.
\end{proof}

The following result allows us to use induction on the cardinality of the support of a fat point scheme. 

\begin{proposition} 
    \label{prop:general}  
Let $Z \subset \PP^n$ be a fat point scheme satisfying $r (Z) \le \seg (Z)$. Then, for every point $P \in \PP^n$ that is not in the support of $Z$ and every integer $m \ge 1$,  one has $r(Z + m P) \le \seg (Z+m P)$. 
\end{proposition}

\begin{proof} 
We use induction on $m \ge 1$.  If $m=1$, then we are done by Theorem \ref{thm:add one point}. 

Let $m \ge 2$. We want to apply Inductive Technique 2 to $X = Z + m P$, where $Z = \sum_{j=1}^s m_j P_j$. This requires some preparation. Set $\sigma = \seg (X)$ and consider the vector matroid associated to the matrix 
\[
A_Z = \oplus_{i=1}^s [P_i]^{m_i} 
\]
with column set $E_Z$. We may assume that the support $\supp (Z)$ of $Z$ is not contained in a hyperplane of $\PP^n$. Thus, this matroid has rank $n+1$. Define a matroid $M$ on $E_Z$ whose rank function is defined by $\rk_M (S) = \rk (S+P) - 1 = \dim \spn (S + P)$ for any subset $S \subseteq E_Z$. Thus, we get 
\[
\rk_M (S) \ge  \dim \spn (S) = \rk (S) - 1. 
\]
In particular, a subset $I$ of $E_Z$ is independent in $M$ if and only if $I+P$ is a  linearly independent subset of $\PP^n$. 
We now argue that, for every subset $S \neq \emptyset$ of $E_Z$, one has 
\begin{align}
   \label{eq:card est}
|S| \le \sigma \cdot \rk_M (S) - (m-1).   
\end{align}
Indeed, given any subset $S \neq \emptyset$ of $E_Z$, extend $S$ by $m$ copies of $P$ to  a subset $S'$ of $E_X$.  Then one has $\rk S' \ge 2$, and thus by applying Lemma \ref{lem:card estimate} to $S'$  we obtain 
\begin{align*}
|S| + m = &  |S'|  
 \le \sigma \cdot \{ \rk (S') - 1 \} + 1 
  = \sigma \cdot \rk_M (S) + 1, 
\end{align*}
which completes the argument for Estimate \eqref{eq:card est}. 

We are now going to show the following key statement. 
\smallskip

\emph{Claim}: There are $t = \binom{n+m-2}{n-1}$ generators $g_1,\ldots,g_t$ of $I_P^{m-1}$ and degree $\sigma - m+1$ forms $f_1,\ldots,f_t$ with $f_j (P) \neq 0$ such that 
\begin{equation}
    \label{eq:key containment}
    g_j f_j \in I_{Z + (m-1) P} \quad \text{ for } j = 1,\ldots,t. 
\end{equation}
\smallskip

\noindent
To establish this claim, we use induction on $m \ge 1$. Let $m =1$. Then Estimate \eqref{eq:card est} is also true for $S = \emptyset$. Hence   Corollary \ref{cor:edmonds} gives a partition $E_Z = I_1 \sqcup \ldots \sqcup I_{\sigma}$ into independent sets of $M$. Thus, $P$ is not in any $\spn (I_j)$, and so there are $\sigma$ linear forms $\ell_j$ such that $\ell_j (P) \neq 0$ and $I_j \subset H_j$, where $H_j$ is the hyperplane defined by $\ell_j$. It follows that $f = \ell_1 \cdots \ell_{\sigma}$ is in $I_Z$ and $f (P) \neq 0$, as desired. 

Let $m \geq 2$. Choose a point $Q_1 \in \PP^n \setminus \{P\}$. Pass from the vector matroid to the matrix $A_Z \oplus [Q_1]$ to a matroid 
 $\widetilde{M}$ on $E_Z \cup \{Q_1\}$ as for $M$ above. That is, $\rk_{\widetilde{M}} (S) = \rk (S+P) - 1 = \dim \spn (S + P)$ for any subset $S \subseteq E_Z \cup \{Q_1\}$. 
Estimate \eqref{eq:card est} shows that we can apply  Corollary~\ref{cor:inductive step} to  obtain a partition 
\[
E_Z = I_1 \sqcup J_1, 
\]
where $I_1$ is independent in $M$, $Q_1 \notin \spn (I_1 + P)$, and 
\begin{equation}
   \label{eq:may apply IH}
|B| \le (\sigma -1) \cdot \rk_M (B) - (m-2)  
\end{equation} 
for each subset $B \neq \emptyset$ of $J_1$. Let $W_1$ be the fat point scheme determined by $J_1$, that is, $W_1 = \sum_{j=1}^s n_j P_j$, where $n_j$ is the number of column vectors in $J_1$ corresponding to the point $P_j$. Estimate~\eqref{eq:may apply IH} shows that the induction hypothesis applies to $W_1$. Hence, there are $u = \binom{n+m-3}{n-1}$ generators $h_1^{(1)},\ldots,h_u^{(1)}$ of $I_P^{m-2}$ and degree $\sigma - m +1$  forms $q_1^{(1)},\ldots,q_u^{(1)}$ with $q_j^{(1)} (P) \neq 0$ such that $h_j^{(1)} q_j^{(1)} \in I_{W_1 + (m-2) P}$ for each $j$.  

Since $Q_1$ is not in the span of the linearly independent set $I_1 + P$, there is a linear form $\ell_1$ such that $\ell_1 (Q_1) \neq 0$ and $I_1 + P \subset H_1$, where $H_1$ is the hyperplane defined by $\ell_1$. Taking into account that $E_Z = I_1 \sqcup J_1$, it follows that $\ell_1 h_j^{(1)} q_j^{(1)} \in I_{Z + (m-1) P}$ for each $j$. 

Notice that this construction works for any point in $\PP^n \setminus \{P\}$. Repeating it $(n-1)$ more times by choosing alltogether points $Q_1,\ldots,Q_n \in \PP^n \setminus \{P\}$, we obtain linear forms $\ell_1,\ldots,\ell_n \in I_P$ as well as $n$ generating sets $\{h_1^{(i)},\ldots,h_u^{(i)}\}$ of $I_P^{m-2}$, and degree $\sigma - m +1$ forms $q_j^{(i)}$ with $q_j^{(i)} (P) \neq 0$ such that 
\begin{equation}
   \label{eq:contain}
\ell_i h_j^{(i)} q_j^{(i)} \in I_{Z + (m-1) P} \quad \text{ for all } i = 1,\ldots,n, \ j = 1,\ldots,u. 
\end{equation} 
We claim that by choosing the points $Q_2,\ldots,Q_n$ suitably  we can additionally achieve that the linear forms $\ell_1,\ldots,\ell_n$ are linearly independent. We show this recursively. Let $2 \le i \le n$ and assume that points $Q_1,\ldots,Q_{i-1}$ have been found such that the linear forms $\ell_1,\ldots,\ell_{i-1}$ are linearly independent. Let $H_j$ be he hyperplane defined by $\ell_j$.  Since  $\dim (\bigcap_{j=1}^{i-1} H_j ) \ge 1$, there is a point $Q_i$ in $(\bigcap_{j=1}^{i-1} H_j) \setminus \{P\}$. By construction of $H_i$, the point $Q_i$ is not contained in $H_i$. Thus, we get 
\[
\dim \bigcap_{j=1}^{i} H_j  = \dim \bigcap_{j=1}^{i-1} H_j  - 1 = n- (i-1) - 1 = n-i. 
\]
In particular, we have shown that $\dim (\bigcap_{j=1}^{n} H_j ) = 0$. Since each of the hyperplanes $H_j$ contains the point $P$, we conclude that the ideal of this point is 
$I_P = (\ell_1,\ldots,\ell_n)$. Now it follows that $\{\ell_i h_j^{(i)} \; | \; 1 \le i \le n, \ 1 \le j \le u\}$ is a generating set of $I_P \cdot I_P^{m-2} = I_P^{m-1}$. Together with the Containment~\eqref{eq:contain}, this establishes the claim. 
\smallskip 

For the remainder of the argument, adopt the notation of the above claim.  Since each form $g_j$ vanishes  precisely to order $m-1$ at $P$, it follows that $I_{Z + (m-1) P} : f_j g_j = I_P$, and thus 
\[
r (Res_{g_j f_j} (Z + mP))  = r(P) = 0
\]
for each $j$.  Since $Z+ (m-1) P$ is a subscheme of $Z + mP$, the definition of the Segre bound implies $\seg (Z+ (m-1) P) \le \seg (Z + m P) = \sigma$. By the induction hypothesis on $m$, we know $r (Z+ (m-1) P) \le \seg (Z+ (m-1) P)$, and so we get  $r (Z+ (m-1) P) \le \sigma$. Thus, applying Lemma~\ref{lem:indTech 2} we conclude that $r (Z + mP) \le \sigma$, as desired. 
\end{proof}

The regularity bound announced in the introduction follows now easily. 

\begin{theorem} 
     \label{thm:Segre bound}
If $X$ is any fat point subscheme of  $\bb{P}^n$, then
  $r(X) \le \seg(X)$. 
\end{theorem}

\begin{proof} 
This is true if $X$ consists of one point (see Remark \ref{rem:Segre single point}). Thus, we conclude by induction on the cardinality of $\supp X$  using the above proposition.
\end{proof}

We conclude by discussing a modification of the above Segre bound. To this end consider the $d$-th Veronese embedding  $v_d:\bb{P}^n \to \bb{P}^{N}$,  where $d \in \N$ and $N = {n+d \choose d} - 1$. We use it to compare the regularity indices of fat point schemes in $\PP^n$ and $\PP^N$, respectively. 

\begin{proposition}
   \label{prop:Veronese}
Let $X = \sum_{i=1}^s m_i P_i$ be a fat point subscheme of $\PP^n$.  Define a fat point subscheme $\hat{X}$ of $\PP^N$ by $\hat{X} = \sum_{i=1}^s m_i  v_d (P_i)$. Then one has 
${\displaystyle 
\left \lceil \frac{r (X)}{d} \right \rceil \le r (\hat{X}). }$

Moreover,  if both $n = 1$ and $d (m_j + m_k) \le 2 d - 2 + \sum_{i=1}^s m_i$ for all integers $j, k$ with $1 \le j < k \le s$, then this is an equality and  $r (\hat{X}) = 
\left \lceil \frac{-1 + \sum_{i=1}^d m_i}{d} \right \rceil$. 
\end{proposition}

\begin{proof}
Let  $S = \oplus_{j \in \N_0} [R]_{j d}$ be the $d$-th Veronese subring of $R = K[x_0,...,x_n]$. It is a polynomial ring in variables $y_a$, where $y_a$ corresponds to the monomial $x^a = x_1^{a_1}  \cdots x_n^{a_n}$ of degree $d$. Consider the ring homomorphism $\ffi: S \to R$ that maps $y_a$ onto $x^a$. Observe that, for each point $P \in \PP^n$, one has $\ffi (I_{v_d (P)}) \subset I_P$. If follows that $\ffi (I_{\hat{X}}) \subset I_X$, and so $I_{\hat{X}} \subset \ffi^{-1} (I_X)$. Furthermore, the ideal $\ffi^{-1} (I_X)$ of $S$ is saturated. Indeed, if $f \in S$ is a homogeneous polynomial that multiplies a power, say, the $k$-th power of the ideal generated by all the variables in $S$ into  $\ffi^{-1} (I_X)$, then $\ffi(f) \cdot  (x_0,...,x_n) ^{k d} \subset I_X$. Since $I_X$ is saturated, this implies $f \in \ffi^{-1} (I_X)$, as desired. 

Thus, the ideal $\ffi^{-1} (I_X)$ is the homogenous ideal of a zero-dimensional subscheme $W \subset \PP^N$, and one has 
\[
H^1 (\PP^n, \Ic_X (j)) \cong H^1 (\PP^N, \Ic_W (j d)). 
\]
Hence, Lemma \ref{lem:reg subscheme}(a) implies $r (W) = \left   \lceil \frac{r (X)}{d}  \right \rceil$. Since $W$ is a subscheme of $\hat{X}$, Lemma \ref{lem:reg subscheme}(b) gives $r (W) \le r(\hat{X})$, and now the first assertion follows. 

In order to show the second claim, assume $n = 1$. Thus  $N = d$, and  $\supp \hat{X}$ lies on a rational normal curve of $\PP^d$. It follows that the support of $\hat{X}$ is in linearly general position, that is, any subset of $j+1 \le d+1$ points span a $j$-dimensional linear subspace of $\PP^d$. Therefore, a straightforward computation shows that the Segre bound of $\hat{X}$ is determined by the one-dimensional subspaces and $\PP^d$, that is, 
\[
\seg \hat{X} = \max \left \{m_j + m_k - 1, \left \lceil \frac{-1 + \sum_{i=1}^s m_i}{d} \right \rceil  \; | \; 1 \le j < k \le s \right \}. 
\]
Combining the assumption and Theorem \ref{thm:Segre bound}, we obtain 
\[
r (\hat{X}) \le \seg \hat{X} = \left \lceil \frac{-1 + \sum_{i=1}^s m_i}{d} \right \rceil. 
\]

Since $X$ is a subscheme of $\PP^1$, its homogeneous ideal is a principal ideal of degree $\sum_{i=1}^s m_i$. Thus, $r (X) = - 1 +  \sum_{i=1}^s m_i$. Now the first assertion gives the desired equality. 
\end{proof} 

As a first consequence, we describe instances where the Segre bound in Theorem~\ref{thm:Segre bound} is sharp. The result extends \cite[Proposition 7]{CTV}. 

\begin{corollary}
     \label{cor:sharpness}
Let $X \subset \PP^n$ be a fat point subscheme, and let $L \subset \PP^n$ be a positive-dimensional linear subspace such that $\seg X =  \left \lceil \frac{w_L (X) -1}{\dim L} \right \rceil$. If the points of $\supp X$ that are in $L$ lie on a rational normal curve of $L$, then $r(X) = \seg X$.     
\end{corollary}

\begin{proof}
Consider the fat point subscheme  $Y = \sum_{P_i \in L} m_i P_i$  of $X$ such that $w_L (X) = w_L (Y)$. If $\dim L = 1$, then $w_L (Y) - 1 = r(Y) \le r(X) \le w_L (X) - 1$, and thus the claim follows. 

Assume $\dim L \ge 2$. Considering lines through any two points in the support of $X$,  the assumption on $L$ gives $m_j + m_k - 1 \le  \left \lceil \frac{w_L (Y) -1}{\dim L} \right \rceil$ for all $j < k$. Hence, applying Proposition \ref{prop:Veronese} with $\hat{X} = Y$, we conclude $r (Y) =  \left \lceil \frac{-1 + \sum_{P_i \in L} m_i}{\dim L} \right \rceil =   \left \lceil \frac{w_L (Y) -1}{\dim L} \right \rceil = \seg X$. Since $r (Y) \le r(X)$, the desired equality follows by Theorem \ref{thm:Segre bound}. 
\end{proof} 

The second consequence of Proposition \ref{prop:Veronese} is an alternate regularity bound. Notice that the following result specializes to Theorem \ref{thm:Segre bound} if $d=1$. 

\begin{proposition}
    \label{prop:modified Segre bound}
Given any scheme of fat points $X = \sum_{i=1}^s
  m_iP_i \subseteq \bb{P}^n$ and any integer $d \geq 1$, the regularity index of $X$  is
  subject to the bound 
\[
r(X) \le \max\left\{d \cdot \left\lceil\frac{-1 + \sum_{P_i \in Y} m_i}{\dim_K [R/I_Y]_d -1} \right \rceil \; | \; \; Y \subseteq \supp X  \text{ and } |Y| \ge 2 \right\}. 
 \]
\end{proposition}

\begin{proof}
Consider the $d$-th Veronese embedding $v_d:\bb{P}^n \to
  \bb{P}^{N}$. As above, let  $R$ and $S$ be the coordinate rings of $\PP^n$ and $\PP^N$, respectively. Notice that the Segre bound of $\hat{X} = \sum_{i=1}^s m_i v_d (P_i)$ is 
\[
\seg \hat{X} = \max \left \{\left \lceil \frac{-1 + \sum_{v_d (P_i) \in L} m_i}{\dim L } \right\rceil \; | \; \; L \subseteq \PP^N   \text{ linear}, \dim L \ge 1 \right\}. 
\]
Consider a linear subspace $L \subset \PP^N$ for which the right-hand side above is maximal. Set $Y = \{P_i \in \supp X \; | \; v_d (P_i) \in L\}$. The assumption on $L$ gives that $\hat{Y} = v_d (Y)$ is not contained in a proper subspace of $L$, that is, $\dim_K [S/I_{\hat{Y}}]_1 -1 = \dim L$. Since $\dim_K [S/I_{\hat{Y}}]_1 = \dim_K [R/I_Y]_d$, Theorem \ref{thm:Segre bound} gives 
\[
r (\hat{X}) \le \seg \hat{X} = \left\lceil\frac{-1 + \sum_{P_i \in Y} m_i}{\dim_K [R/I_Y]_d -1} \right \rceil. 
\]
Using $\frac{r(X)}{d} \le r(\hat{X})$ due to  Proposition \ref{prop:Veronese}, the claim follows. 
\end{proof}

If one has information on subsets of the points supporting a fat point scheme, then the above result can be used to obtain a better regularity bound than the Segre bound of Theorem~\ref{thm:Segre bound}. We illustrate this by a simple example. 

\begin{example}
    \label{exa:mod bound}
Let $X = \sum_{i=1}^s m P_i \subset \PP^n$ be a fat point scheme, where all points have the same multiplicity $m$. Suppose that the support of $X$ consists of five arbitrary points and $\binom{d+n}{n}$ generic points for some $d \ge 5$. Thus, $s = 5 + \binom{d+n}{n}$. Let $L \subset \PP^n$ be a linear subspace of dimension $k$ with $1 \le k < n$. Then $| L \cap \supp X| \le k+4$. It follows that for sufficiently large $d$ (or $n$) 
\[
\seg X = \max\left\{\left \lceil \frac{(k+4) m - 1}{k}\right\rceil, \left \lceil \frac{[\binom{d+n}{n} + 5] m - 1}{n} \right\rceil  \; | \; 1 \le k < n \right \} = \left \lceil \frac{\binom{d+n}{n} m + 5m - 1}{n} \right\rceil. 
\]

Consider now any subset $Y \subset \supp X$ of $t \ge 2$ points. Since $d \ge 5$, one gets 
\[
\dim_K [R/I_Y]_d = \begin{cases} 
t & \text{ if } t \le \binom{n+d}{n} \\
\binom{n+d}{n} & \text{ otherwise}. 
\end{cases}
\]
Hence, Proposition~\ref{prop:modified Segre bound} and a straightforward computation give
\begin{align*}
r (X) & \le d \cdot \max\left\{\left \lceil \frac{t m - 1}{t-1}  \right\rceil, \left \lceil \frac{[\binom{d+n}{n} + 5]  m - 1}{\binom{d+n}{n} - 1} \right\rceil \; | \; 1 \le t \le \binom{d+n}{n} \right \}  \\
& =  d \cdot \max \left\{2 m -1, \left \lceil \frac{[\binom{d+n}{n} + 5] m - 1}{\binom{d+n}{n} - 1} \right\rceil \right \}. 
\end{align*} 
For sufficiently large $d$ (or $n$), this implies $r (X) \le d (2m-1)$. In comparison, $\seg X$ is essentially a polynomial function in $d$ of degree $n$. 
\end{example}


\end{document}